\documentclass[11pt]{amsart}
\usepackage[T1]{fontenc}
\usepackage[ansinew]{inputenc}
\usepackage{amssymb,amsmath,amsthm,eucal,hyperref,geometry,mathtools,mathrsfs,array,graphicx,xcolor,bbm,relsize,enumerate,fullpage,subcaption,comment}

\usepackage[export]{adjustbox}
\usepackage[all]{xy}
	\usepackage[displaymath, mathlines]{lineno}

\usepackage{bbm}
	\usepackage{mathtools}
	\usepackage[T1]{fontenc}
	\usepackage{color}
	\usepackage{textcomp}
	
	 \numberwithin{equation}{section}

	\DeclareMathSymbol{e}{\mathalpha}{operators}{`e}
	
	\DeclareMathOperator{\var}{Var}

	\renewcommand{\to}{\rightarrow}
	\renewcommand{\L}{\mathcal{ L}}
	\newcommand{\R}{\mathbb{R}}
	
	\newcommand{\CC}{\mathcal C}
	\newcommand{\BB}{\mathcal B}
	
	\newcommand{\N}{\mathbb{N}}
	\newcommand{\Z}{\mathbb{Z}}

	\renewcommand{\P}{\mathbb{P}}  
	
	\renewcommand{\AA}{\mathcal A}
	
	\newcommand{\F}{\mathscr{F}} 
	
	\newcommand{\E}{\mathbb{E}} 
	\newcommand{\Var}{\mathrm{Var}}

	\renewcommand{\H}{\mathbb H}
	\newcommand{\HH}{\mathcal H}

	\renewcommand{\1}{\mathbf{1}}
	\newcommand{\I}[1]{\mathbf{1}_{\left \{#1\right \}}}
	
	\renewcommand{\leq}{\leqslant}
	\renewcommand{\geq}{\geqslant}
	\renewcommand{\d}{{d}}
	 
	\newcommand{\eps}{\varepsilon}
	\renewcommand{\SS}{\mathcal S}
	\newcommand{\TT}{\mathcal T}

	\let\epsilon\varepsilon
	\let\ln\log
	
	\newtheorem{defi}{Definition}[section]

	\newtheorem{lema}[defi]{Lemma}
	\newtheorem{claim}[defi]{Claim}
	\newtheorem{prop}[defi]{Proposition}
	\newtheorem{cor}[defi]{Corollary}

	\newtheorem*{rem}{Remark}

\begin{document}
	
	\title{On thin local sets of the Gaussian free field}
	\author{Avelio Sep\'ulveda}
	\thanks{This work was supported by the SNF grant \#155922. The author is part of the NCCR Swissmap.}
\address{Department of Mathematics, ETH Z\"urich, R\"amistr. 101, 8092 Z\"urich, Switzerland} 
\email{leonardo.sepulveda@math.ethz.ch}

	\begin {abstract}

We study how small a local set of the continuum Gaussian free field (GFF) in dimension $d$ has to be to ensure that this set is thin, which loosely speaking means that it captures no GFF mass on itself, in other words, that the field restricted to it is zero. We provide a criterion on the size of the local set for this to happen, and on the other hand, we show that this criterion is sharp by constructing small local sets that are not thin.  
	\end {abstract}

	\maketitle 
	
	\section{Introduction}
	
The Gaussian Free Field (GFF) is the natural analogue of Brownian motion when the time-set is replaced by a $d$-dimensional open domain $D$. The GFF is a fundamental object in probability and statistical physics. In two dimensions, its geometry is closely related to many other key objects such as Stochastic Loewner Evolutions \cite{Dub,SchSh2,MS1}, conformal loop ensembles \cite{MS,ASW2},  Liouville quantum gravity \cite{DS,JA2,APS}, quantum Loewner evolutions \cite{MSQLE,MSQLE2} and Brownian loop soups \cite{LJ, Titus, QW,ALS2}; note that the relation to loop soups is
in fact not restricted to the two-dimensional GFF.

 Unlike Brownian motion, when $d \ge 2$, the GFF is not a continuous function; it can only be defined as a random generalised function from $D$ into $\R$. However, the GFF has many properties analogue to those of the Brownian motion. In particular, it has a spatial Markov property. The spatial Markov property of the GFF states that for any deterministic closed set $A$ the distribution of the GFF in $D\backslash A$ is equal to the sum of the harmonic extension of the values of the GFF on $\partial A$, and an independent GFF in $D \setminus A$. Just as in the one-dimensional case, this Markov property can be upgraded into a strong Markov property, where the above decomposition holds also for some random sets $A$.
 Such multivariate Markov properties were first studied in the 70s and 80s \cite{Roz}, and recently reinterpreted and applied in the two-dimensional imaginary geometry framework \cite{SchSh2,MS1}. These sets, called local sets in \cite {SchSh2, MS1}, play roughly the same role, in the higher-dimensional setting, as stopping times; more precisely, the local set $A$ is 
 the analogue of the interval $[0,\tau]$ when $\tau$ is a one-dimensional stopping time. The notion of local sets makes sense and is natural for the GFF in any dimension, even if so far it has only been used  when $d=2$. 

One way to formally describe local sets is to say that there exists a coupling $(\Gamma,A,\Gamma_A)$ where $\Gamma$ is a GFF in $D$, $A$ is a random closed set and $\Gamma_A$ is a random field with the following properties: 
\begin {itemize}
 \item Conditionally on $(A, \Gamma_A)$, the distribution of $\Gamma -\Gamma_A$ is a GFF in $D \setminus A$. 
 \item For every deterministic open set $O$, on the event where $O$ and $A$ are disjoint, the restriction of $\Gamma_A$ to $O$ is 
 a harmonic function in $O$. More precisely, there exists a random harmonic function $h_A$ in $D\backslash A$ such that for all smooth function $f$,
 $(\Gamma_A,f) = \int_{D\backslash A}  h_A (x) f(x) dx $  on the event where the support of $f$ is contained in $D\backslash A$. 
\end {itemize}
The field $\Gamma_A$ can be understood as being equal to the field $\Gamma$ ``within $A$'' and to the harmonic extension $h_A$ 	of the values of the field on $\partial A$ in $D\backslash A$. 

In the present paper, we investigate how small a local set has to be (in terms of its fractal dimension) to ensure that, loosely speaking,   that $\Gamma$ restricted to $A$ is equal to $0$, in other words,  that ``$\Gamma_A=h_A$''.   We call a local set satisfying this property thin local sets.  As the GFF is not a function, the precise definition of thin local sets is not straight-forward, and it is discussed in what follows.

\subsection {Definition of thin local sets}\label{thin1} Let us start with a particular case. Assume that the harmonic function $h_A$ is a.s. integrable on $D\backslash A$\footnote{this for instance happens for the bounded-type thin local sets studied in \cite {ASW2} where $h_A$ is bounded}, being thin means that for any compactly supported smooth function $f$, $(\Gamma_A,f)$ is almost surely equal to $\int_{D\backslash A} h_A(z)f(z)dz$, even when the support of $f$ intersects $A$.

One of the main questions of this paper is to find a good definition of a thin local set when $h_A$ oscillates in the boundary. By this, we mean the case when the function $h_A$ is not integrable on $D \setminus A$. This framework should be thought of as the generic case as $h_A$ tends to oscillate wildly when it approaches $A$. This is especially true in higher dimensions where this is already the case when $A$ is a deterministic non-polar set.

There are many possible definitions for thin local sets, and we will discuss them in the last section of the paper. At this point and for the rest of the paper until Section \ref{Thin sets def}, we will fix a definition based on dyadic approximations.

Suppose that $D$ is a fixed bounded open domain in $\R^d$ for $d \ge 2$. For any $n \ge 0$, say that $s$ is an open dyadic hyper-cube of  side-length $2^{-n}$ (or just $2^{-n}$ dyadic hypercubes) if it is a translate of $(0,2^{-n})^d$ by some element in $(2^{-n}\Z)^d$. 
We call $\SS_n$ the set of all non-empty intersections of open $2^{-n}$-dyadic hypercubes with $D$ and $\TT_n$ the set of faces of elements of $\SS_n$. If $A$ is a closed set, we define $A_n$ to be the closure of the union of all elements of $\SS_n\cup \TT_n$ intersecting $A$. 

Let us note that for any closed set $A$, $A_n$ decreases to A and that for all $n\in \N$, $A_n$ can take only finitely many values. This allows us to define, for each smooth function in $D$, the random variable $(\Gamma_A,f\1_{D\backslash A_n})$. Indeed, one can simultaneously define $(\Gamma_{A},f\1_u)$ for any possible value $u$ of $D\backslash A_n$ , and then see that $ (\Gamma_A, f \1_{D\backslash A_n} )$ is a.s. equal to $\sum_n  (\Gamma_{A},f\1_u)\I{D\backslash A_n =u }$.

\begin {defi}[Thin local set]
A local set $A$ is a thin local set if for any smooth bounded function $f$ in $D$, the sequence of random variables $(\Gamma_A,f1_{D\backslash A_n} )$ converges in probability 
to $(\Gamma_A,f)$ as $n \to \infty$.
\end {defi}

The intuition behind this definition is that the limit of this sequence of random variables should be thought of as a way to make sense of $(\Gamma_A,f1_{D\backslash A})$, which then has to be the same as $(\Gamma_A,f)$. 

We leave it as an exercise to check that in the particular case of local sets where $h_A$ is integrable, this definition is equivalent to the fact that a.s. $(\Gamma_A,f)=\int_{D\backslash A} h_A(z) f(z)dz$. To do this, first one has to check that for all possible values of $u$ of $D\backslash A_n$, $(\Gamma_{A},f\1_u)$ is a.s. equal to $\int_{u} h_A(z)f(z) dz$.

Finally, let us note  that the choice of working with dyadic approximations is somewhat arbitrary and the question whether changing this choice would change the definition is in fact open. Additionally, even though the examples of non-thin local sets that we will describe in Section \ref{hints} are tailor-made for this particular approximation; it is easy to adapt them to many other analogous choices. We will comment further on this in Section \ref{Thin sets def}.

\subsection{Results}  The results of this paper quantify how small may be non-thin local set. For instance, a deterministic set is a thin local set if and only it has zero Lebesgue measure. However, as we shall see, when $d \ge 2$ there exist many (random) non-thin local sets that have zero Lebesgue measure. In some sense, this is because the GFF values can be explored in a way that captures large values of the GFF while keeping the explored set local and relatively small.

Let us briefly present our main results first when $d \ge 3$ and then $d=2$.
\begin{prop} Let $d\geq 3$ and $\Gamma$ be a GFF in $D\subseteq \R^d$, then\hypertarget{1d}:
	\begin{enumerate}
		\item [(1)$_d$]If $A$ is a local set of a $d$-dimensional GFF and a.s. has upper Minkowski dimension strictly smaller than $\max \{ 1 + (d/2),d-2 \} $, then it is thin\hypertarget{2d}.
		\item [(2)$_d$]There exist local sets of the GFF such that with positive probability their upper Minkowski dimension is equal to  $\max \{ 1 + (d/2),d-1 \} $ that are not thin local sets.
		\end {enumerate}
	\end{prop} 
The two different upper bounds in \hyperlink{1d}{(1)$_d$} have very different nature. The term $1+(d/2)$ comes from the fact that, because of the nature of the singularity of the $d$-dimensional Green's function, the variance of the integral of the GFF over an $\epsilon$-ball is of order $\epsilon^{-(d-2)}$. On the other hand, the term $d-2$ is related to the dimension of polar sets in dimension $d$. 

Note that the numbers of \hyperlink{1d}{(1)$_d$} and \hyperlink{2d}{(2)$_d$} match for $d=3,4$. In other words, the dimensions $5/2$ and $3$  play an important role in the size oflocal sets of the GFF in dimensions $d=3$ and $d=4$ respectively. Furthermore we believe that they should match for any $d\geq 3$. Thus, the threshold $(d/2)+1$ would then be valid up to $d=6$, and for $d \geq 6$, it should be $d-2$.

In fact,  \hyperlink{1d}{(1)$_d$} and \hyperlink{2d}{(2)$_d$} also hold in the two-dimensional case. However, the second statement is rather void as $1 + (d/2) = 2$, and to prove it one could just take $A$ to be the entire domain $\overline D$, which is clearly not thin. We derive the following more refined result when $d=2$:
\begin{prop}Let $\Gamma$ be a GFF in $D\subseteq \R^d$, then\hypertarget{12}:
	\begin{enumerate}
		\item[(1)$_2$] If $A$ is a local set of the two-dimensional GFF such that the expected value of the area of the $\eps$-neighbourhood of $A$ decays like  $o(|\log \eps|^{-1/2}  |\log|\log \eps||^{-1/2})$, then it is a thin local set\hypertarget{22}.
		\item[(2)$_2$] There exist local sets of the two-dimensional GFF for which the expected value of the area of their $\eps$-neighbourhood decays like  $O(|\log \eps|^{-1/2})$ and that are not thin local sets. 
	\end{enumerate}
\end{prop}
 As we see, in dimension $d=2$ there is a logarithmic term that appears. This is of no surprise, as the variance of the average GFF over an $\epsilon$-ball is of order $|\log \epsilon|$. It is important to remark that a result of the type of \hyperlink{22}{(2)$_2$} has also appeared in \cite{ALS}, where the authors show that the constructed local set has a non-trivial Minkowski measure with gauge $r\mapsto r^2|\log(r)|^{-1/2}$.
 
 As explained before proofs of statements of the type (1)$_d$ (i.e.  ``when the local set is small enough, then it is necessarily thin'') are based on two ideas. For the upper bound $1+(d/2)$, we use a first moment computation to show that very high values of the GFF are so sparse that they do not give mass. This allows us to assume that $\epsilon$-averages of the GFF are bounded by a certain deterministic function of $\epsilon$. For the upper bound $d-2$, we show that thin local sets which are polar do not give information about the GFF and thus they are thin (see Lemma \ref{polar}).
 
 It is somewhat more challenging to prove \hyperlink{2d}{(2)$_d$}, i.e. to construct well-chosen ``fairly small'' local sets and to prove that they are not thin. This is arguably the main contribution of the present paper. It is worthwhile noticing that in two-dimensions, it is possible to use the nested version of the Miller-Sheffield GFF-CLE$_4$ coupling to construct such a small yet non-thin local set \cite{ALS}, but when $d \ge 3$ other ideas are needed. 
 Our strategy consists in relating a particular exploration of the GFF with a branching Brownian motion. This idea is reminiscent of the one that was for instance used in the two-dimensional case in \cite{Bol-Er-Deu} to study the maximum of the discrete GFF. 
 The constructed set may also be interpreted as a local set approximation of perfect thick points (in the sense of \cite{Hu-Mi-Pe}, Section 3.2). Note that the main difficulty of this part is that the sets we study need to remain thin.
 
 The structure of the paper is the following: first, we briefly recall some fundamental properties of the continuous GFF and its local sets. Then, we  construct examples of local sets that prove the statements \hyperlink{2d}{(2)$_d$}.  After that, we prove the statements \hyperlink{1d}{(1)$_d$} and conclude with some comments about the definitions of thin local sets.

\subsection* {Acknowledgements.} First, I wish to thank my (former) advisor Wendelin Werner for having proposed me the problem, for all the inspiring discussions and specially for all the many times that he read the many manuscripts proposing each time many changes that made the paper (much more) readable. I also wish to thank Juhan Aru, for the inspiring and useful conversations and the (many) times he read this manuscript helping it to improve it every time. Additionally, I would like to thank Ron Rosenthal for the stoicism he showed when reading the very first version of this manuscript and commenting it. Finally, I would like to express my gratitude to an anonymous referee for his/her careful reading and comments.

\section{Preliminaries}

\subsection {GFF and scaling}
Introductions and basic results about the GFF can be found in  \cite{SchSh2,She,JA,WWln2,AS}. While the presentations in those references is in the two-dimensional setting, they can be extended without any difficulty to higher dimensions. Let us briefly remind some basic facts.

Throughout this paper, we use the function $\phi_d$ defined on $\R^d \setminus \{ 0 \}$ by $\phi_d (x) = (1/ 2 \pi) \times \log ( 1 / \| x \|)$ when $d =2$ and by $\phi_d(x) = 1 / (c_d \| x \|^{d-2} )$ when $d \ge 3$, where $c_d$ denote the $d-1$-dimensional surface area of the unit sphere in $\R^d$. 

Suppose that $D$ is $d$-dimensional open domain with non-polar boundary (this boundary can be empty if $d\ge 3$), and consider the Green's function with Dirichlet boundary condition in $D$ to be the unique function from $D \times D \setminus \{(x,x): \ x \in D \}$ to $\R_+$ that is harmonic in both variables, and such that for all given $x$ in $D$, $G_D (x,y) \to 0$ as $y \to \partial D$ and $G_D(x,y) \sim \phi_d(x-y)$ as $y \to x$. Recall that when $D \subset \tilde D$, then $G_D (x,y) \le G_{\tilde D} (x,y)$.

We can then define the space $\HH^{-1} (D)$ of functions on $D$, such that 
$$ \iint_{D\times D} f(x)G_D(x,y)f(y)\d x\d y <\infty .$$
The GFF in $D$ with zero boundary conditions is defined to be the centered Gaussian process $((\Gamma,f), f \in \HH^{-1}(D) )$ with covariance function
$$ \E\left[(\Gamma,f)(\Gamma,g) \right] = \iint_{D\times D}  f(x)G_D(x,y) g(y)\d x \d y.$$

It is well-known that this process exists, and that it is possible to find a version of the GFF such that almost surely, for all $\epsilon >0$, $\Gamma$ can be viewed as an element  of the Sobolev space $\HH^{1/2-d/4-\epsilon}$. Here, $\HH^{1/2-d/4-\epsilon}$ is the dual under the $\L^2$ product of the Sobolev space $\HH^{d/4-1/2+\epsilon}$ (see for instance Section 2.3 of \cite{She}).

The definition of the GFF immediately implies its scaling properties. If we define the domain 
$ z_0 + rD := \{ z_0 + r z \ : \ z \in D \}$,
then
\begin{align}\label{SGF}
	 G_{z_0 + rD} ( z_0 + rx , z_0 + ry ) = r^{2-d}  {G_D (x,y)}
\end{align}
(in two dimensions, a stronger result holds, as the Green's function is conformally invariant), which 
 yields the corresponding scaling properties for the GFF.

	\subsection {Local sets}
	We first very briefly review the definitions of local sets and some of their properties that are relevant for our purposes. This presentation is based in Section 1.3 of \cite{JA}
	
	Denote the family of all closed subsets of $D$ by $\CC(D)$.         Let $\Gamma$ be a GFF in $D$ and $C\in \CC(D)$. One can decompose $\Gamma$ into the sum of two independent processes $\Gamma_C$ and $\Gamma^C$ where almost surely, $\Gamma_C$ restricted to $D \setminus C$  is a harmonic function, and where  $\Gamma^C$ is a GFF in $D\backslash C$. This property is usually referred to as the spatial Markov property of the GFF. One can note that $\Gamma_C$ and $\Gamma^C$ are Gaussian processes that are also generalised functions, with respective covariance given by the Green's functions $G_D - G_{D \setminus C}$ and  $G_{D \setminus C}$. 
	
	Let $(\F_C)_{C\in \CC(D)}$ be a complete outside-continuous filtration indexed by $\CC (D)$. That is to say, $C \mapsto \F_C$ is non-decreasing, the $\sigma$-fields $\F_C$ are all complete with respect to the probability measure that we are working with, and for any decreasing sequence $(C_n)$, one has $\F (  \cap C_n) = \cap \F (C_n)$. We say that the GFF $\Gamma$ is adapted with respect to this filtration if, for all $C$, $\Gamma_C$ is $\F_C$-measurable while $\Gamma^C$ is independent of $\F_C$.  We also say that a  random set $A$ is a local set in the filtration  $(\F_C)$ if for all $C\in \CC(D)$, the event $\{ A \subseteq C \}$ is in $\F_C$. The filtration generated by a GFF $\Gamma$ (or the ``natural filtration'' of $\Gamma$) is the smallest one for which each $\Gamma_C$ is $\F_C$-measurable.

	Let us list a couple of simple facts about local sets, whose properties are immediate consequences of the definition (see Section 1.3 of\cite{JA}): 
	\begin{enumerate}[a)]
		\item If $A$ and $B$ are local with respect to the filtration $(\F_C)$, then $A\cup B$ is also local.
		\item If $(A_n)$ is a family of local sets with respect to the filtration $(\F_C)$, then  $\cap_n ( \overline {\cup_{m \ge n} A_m} )$ is also a local in the same filtration. 
		\item If $A$ is a local set and $\Gamma$ is a GFF adapted to $\F_\cdot$, then there exists a process $\Gamma_A$, such that it is a.s. harmonic in $D\backslash A$, and that conditionally on $(A,\Gamma_A)$, $\Gamma^A:=\Gamma-\Gamma_A$ is a GFF in $D\backslash A$.
	\end{enumerate}
	In the literature, having a coupling $(A,\Gamma)$ satisfying c) is usually used as the definition of local sets (see for instance \cite{SchSh2}). This property is equivalent to the existence of a filtration under which $A$ is a local set, and $\Gamma$ is a GFF. This can be done by defining $\F_C=\sigma(\Gamma_C,A\1_{A\subseteq C},\{A\subseteq C\})$ and using Lemma 3.9 of \cite{SchSh2} to see this satisfies the definitions. The definition it of local via filtration will be handy to show that the examples that we construct are indeed local sets.

	Note that we can represent the restriction of $\Gamma_A$ to $D\backslash A$ as a harmonic function $h_A$ in $D\backslash A$. In other words, there exists a harmonic function $h_A$ in the random domain $D \setminus A$ such that for all smooth function $f$, $(\Gamma_A,f)=\int h_A(z)f(z)dz$ on the event where the support of $f$ is contained in $D \backslash A$. 
	
	Additionally, it holds that when $A$ and $B$ are local sets,  a.s. for all $z$ such that the connected component of $D\backslash A$ containing $z$ is equal to the connected component of $D\backslash B$ containing $z$ we have that  $h_A(z)=h_B(z)$ (see Proposition 1.3.29 of \cite{JA} or \cite{WWln2}).

Let us already point out that local sets have to be big enough to provide any information about the GFF.
\begin{lema}\label{polar}
	Let $\Gamma$ be a GFF on a domain $D$ and $A$ a local set. Then, $\Gamma_A=0$ almost surely if only if $A$ is almost surely polar for Brownian motion on $D$.
\end{lema}
\begin{proof}
	Note that $A$ is polar if and only if $G^D=G^{D\backslash A}$. Then for all smooth function $f$ with bounded support,
	\begin{align*}
	\E\left[(\Gamma_A,f)^2 \right] =\E\left[(\Gamma,f)^2 \right] -\E\left[(\Gamma^A,f)^2 \right],
	\end{align*}
	Given that $G^{D\backslash A}\leq G^D$, we see that $A$ is polar if and only if the right hand side is equal to 0 for all such $f$.
\end{proof}
Recall that Kakutani's Theorem (Theorem 8.2 in \cite{MoPe}) states that one can check whether a set is polar by studying the decay oo the volume of small neighbourhoods of $A$. In particular, when $d \ge 3$, any local set with Minkowski dimension smaller than $d-2$ is polar for the BM, and it is, therefore, a local set with $\Gamma_A =0$. 
\subsection {A simple bound for Gaussian random variables.} To finish the preliminaries we show a simple Gaussian inequality wish will be used in the proof of statements of the type \hyperlink{1d}{(1)$_d$}.
\begin{lema}\label{inequality gaussian}
	There exists an absolute constant $W$ such that for  any centred Gaussian vector $(X,Y)$ and for all $A$ with $\P(A)\leq \frac{1}{8}$ we have that
	\begin{align}\label{GB}
	\E\left[XY \1_{A} \right]\leq W \max\{\var(Y),\var(x) \} \P(A)\log\left |\P(A)\right|.
	\end{align}
\end{lema}

\begin{proof}
	Using the fact that $2ab\leq a^2+b^2$, we can restrict ourselves to the case where $X=Y$, and by scaling it suffices to consider the case where $X$ is a standard normal variable.  Now, take $r>0$ such that $\P(|X|>r)= \P(A)$ and note that $\E\left[X^2(\1_{|X|>r}-\1_{A})\right]\geq 0$. Thus,
	\begin{align*}
	\E\left[X^2\1_A\right] 
	&\leq\frac{1}{\sqrt{2\pi }}\int_{r}^{\infty} x^{2} e^{\frac{-x^2}{2}}\leq r e^{\frac{-r^2}{2}}.
	\end{align*}
	
	Now, we just need to estimate the value of $re^{-\frac{r^2}{2}}$. To do this we use that for any $s>\frac{2}{\sqrt{2\pi}}$, (see \cite{Co}):
	\begin{align*}
	\frac{2}{\sqrt{2\pi}}\frac{s}{s^2+1}e^{-\frac{s^2}{2}} \leq \P(|X|>s) \leq e^{-\frac{s^2}{2}}.
	\end{align*}
	From the first inequality we get that
	\begin{align*}
	re^{-\frac{r^2}{2}}\leq \frac{\sqrt{\pi}}{\sqrt{2}}\P(|X|>r) (r^2+1),
	\end{align*}
	but from the second inequality we get that $r\leq \sqrt{2|\ln (\P(X>r))| }$. From where we conclude.
\end{proof}

\section{Examples of ``small'' non-thin local sets.}\label{hints}

In the present section, we  prove the statements \hyperlink{2d}{(2)$_d$}: We construct and describe the main features of a particular local set of the $d$-dimensional GFF in $d \ge 2$, which is not thin, yet rather small. 

\subsection {An example using CLE$_4$ in two dimensions}

Before we construct our actual examples, let us quickly describe how it is possible to use the coupling of the two-dimensional GFF with the conformal loop ensemble CLE$_4$ to construct a local set which implies the statement \hyperlink{2d}{(2)$_2$}. Because such a relationship is only known in dimension $2$, this construction can not be generalised to higher dimensions. However, it helps understanding some features of the example presented in the next subsection. 
Since this CLE$_4$-based construction is not used in our main proofs, we choose here not to give a complete review of the Miller-Sheffield coupling of the CLE$_4$ with the GFF in two dimensions, and we refer the reader to \cite {ASW2} for background and details.

Let $\Gamma$ be a GFF in a simply connected domain $D$. Recall that (see \cite{MS,ASW2}) it is possible to define deterministically from $\Gamma$ a local set $A_1$ of Minkowski dimension $15/8$ (see \cite{SchShW,NW}) such that the harmonic function $h_{A_1}$ (that we denote by $h_1$) is constant and equal to $\pm 2\lambda=\pm\sqrt {\pi /2}$ in each connected component of $D \setminus A_1$\footnote{The constant $2\lambda$ is called the height-gap of the GFF and it depends on the normalisation of the Green's function, and therefore of the GFF. Sometimes, other normalizations are used in the literature: If $G_D (x,y) \sim c \log(1/|x-y|)$ as $x \to y$, then $\lambda$ should be taken to be $(\pi/2)\times \sqrt {c}$. }. This set $A_1$ has the law of a CLE$_4$ carpet, and the coupling just described is usually called the natural coupling of CLE$_4$ with the GFF.

Furthermore, as explained in \cite {ASW2}, this local set is thin (in the present case, the definition of thin is the one given in the introduction because $h_{1}$ is integrable) and conditionally on $A_1$, the sign of $h_1$ is chosen to be $+$ or $-$ independently in each connected component of $D \setminus A_1$. 

Now, we define inductively an increasing family $A_n$ of local sets. Suppose that for a given $n\in \N$, we have defined a certain thin local set $A_n$ such that $h_n$ is constant in each connected component of $D \setminus A_n$ and takes values in  $\{-2 k\lambda:-1\leq k\in \Z\}$. We then define $A_{n+1}$ and $h_{n+1}$ as follows: 
\begin {itemize}
\item In the connected components of $D \setminus A_n$ where $h_n = 2 \lambda$ we do nothing: these connected components are still in $D \setminus A_{n+1}$ and $h_{n+1} = 2 \lambda$ there. 
\item In the other connected components, $O$, of $D \setminus A_n$, we construct the CLE$_4$ associated to the GFF $\Gamma^{A_n}$ restricted to $O$. The connected components of $D \setminus A_{n+1}\cap O$ are defined to be the complement of this CLE$_4$, and the values of the harmonic function are $h_{n+1} = h_n \pm 2 \lambda$.
\end {itemize}
We finally define our local set $A$ to be the closure of $\cup_n A_n$. 

As $A$ is the closure of the union of local sets, $A$ is also a local set. Furthermore, note that for every for every $x\in D$, $(2\lambda)^{-1}h_{A_n}(x)$ is a simple random walk stopped at $N_x$,the first time it hits $1$.  Because for every $x\in D$ a.s. $N_x<\infty$,  we have that the Lebesgue measure of $A$ is 0. Using the techniques of Proposition 20 of \cite{ASW2} or Proposition 4.6 of \cite{ALS} one could further show that $A$ satisfies \hyperlink{22}{(2)$_2$}.

 We can also see that $h_A$ is equal to $2 \lambda$ in each of the connected components of the complement of $A$. To do this it suffices to take a dense set $(x_n)_{n\in \N}$, and note that a.s. $h_{A}(x_n)=h_{A_{N_{x_n}}}(x_n)=2\lambda$.

Since $h_A = 2 \lambda$, the set $A$ can not be thin. Indeed, for any smooth non-negative test function $f$, the integral $\int_{D \setminus A} h_A(z) f(z) dz$ would be 
almost surely non-negative, and it can therefore not be the same random variable as $(\Gamma ,f)-(\Gamma^A,f)$ (unless $f=0$).

To finish this example, let us mention that we know much more about the size of $A$. In \cite{ALS}  it is shown that $A$ has non trivial Minkowski content of gauge $r\mapsto r^2/\sqrt{|\log(r)|}$, and that this Minkowski content is exactly the difference between $(h_A,1)$ and $(\Gamma_A,1)$.

\subsection {Another example in two dimensions} 

 In the present subsection, we first describe another local set $A$ of the two-dimensional GFF that has a simple generalisation when $d \ge 3$. One main feature is reminiscent of the previous case: we discover the GFF in a self-similar fashion and explore the GFF until its mean value in the dyadic square that we are currently looking at is likely to be positive in some sense that we will make precise. The main difference with the previous case is that we explore using boundaries of dyadic squares instead of a nested CLE$_4$, as the CLE technology is not available in higher dimensions.

\medbreak
\noindent
{\bf Notation.}
Choose the domain $D$ to be the unit square $(0,1)^2$. As we are going to use nested dyadic squares, it is useful to introduce the following notation. We define $S^\emptyset$ to be equal to $D$, and when $u$ is a finite sequence of $n$ elements of $\{ 1, \ldots , 4 \}$, then $S^{u1}, \ldots, S^{u4}$ are the four open dyadic subsquares of side-length $2^{-n-1}$ of $S^u$. We can, for instance, choose to associate the four indices respectively to the NW, NE, SW, SE subsquares. Thanks to this notation, we can associate to each square a point in the tree $\bigcup_{n}\{1,2,3,4\}^n$, and a genealogy.

Let us also define for each dyadic square $S^u$, the random variable $\gamma_n (S^u) := (\Gamma_{T_n}, \1_{S^u} )$, where $T_n$ is the union of elements in $\TT_n$. This is the conditional expectation of $(\Gamma, \1_{S^u} )$  in $S^u$, when one observes the GFF outside (i.e. on the boundary) of the ancestor of $S^u$ with height $n$ if $n\leq |u|$(the height of $u$), or the boundary of the children of $u$ with height $n$ if $n>|u|$ . It can also be viewed as $(\Gamma,\mu^u)$ where $\mu^u_n$ is a well-chosen measure supported on the boundary of the squares associated with $S^u$ with height $n$.

We are going to discover progressively and simultaneously the GFF along the four segments from $(1/2, 0)$, $(1, 1/2)$, $(1/2,1)$ and $(0,1/2)$ to the middle point $(1/2, 1/2)$ (see the first image of figure Figure \ref{A}). When we have finished, then the unit square is divided into the four squares $S^1, \ldots, S^4$ of side-length $1/2$. During this discovery, we can choose a modification of the conditional expectation of the random variable $(\Gamma,1)$, given the discovered values of the GFF in the discovered segments, so that it evolves like a continuous martingale. Thus, we can parametrise time in a way such that this conditional expectation has the law of a Brownian motion\footnote{One could show that $B$ is continuous even though it is not needed in the paper. To do this one may use the fact that the trace of the field in a line can be seen as a distribution in $\H^{-1}$ (see Section 4.3.2 of \cite{Dub})} $B = B^\emptyset$ stopped at a time $T$.

 Let us note that the given change of time is not random. To prove this take a deterministic way symmetrically growing the segments $l(t)$. The weak Markov property implies that for any $t\geq0$,  $(\Gamma,1)$ is the sum of $(\Gamma_{l(t)},1)$ and $(\Gamma^{l(t)},1)$. Also, let us note that $(\Gamma,1)$ and $(\Gamma^{l(t)},1)$ are a centred Gaussian with variance $\iint_{D\times D} G_D(x,y)dxdy$ and $\iint_{D\times D} G_{D\backslash l(t)}(x,y)dxdy)$ respectively. Furthermore, as $\Gamma_{l(t)}$ is independent of $\Gamma^{l(t)}$, we have that $(\Gamma_{l(t)},1)$ is distributed as a centred Gaussian random variable with variance $\sigma_t=\iint_{D\times D} (G_D(x,y)-G_{D\backslash l(t)}(x,y))dxdy$. Thus, it suffices to take $l(t)$ such that $\sigma_t=t$. Let us note that this discussion also implies that $T$ is equal to $\iint_{D\times D} (G_D(x,y)-G_{D\backslash T_1}(x,y))dxdy$.

\begin{figure}[ht!]
	\begin{subfigure}{0.4\textwidth}
		\includegraphics[width=0.9\linewidth ]{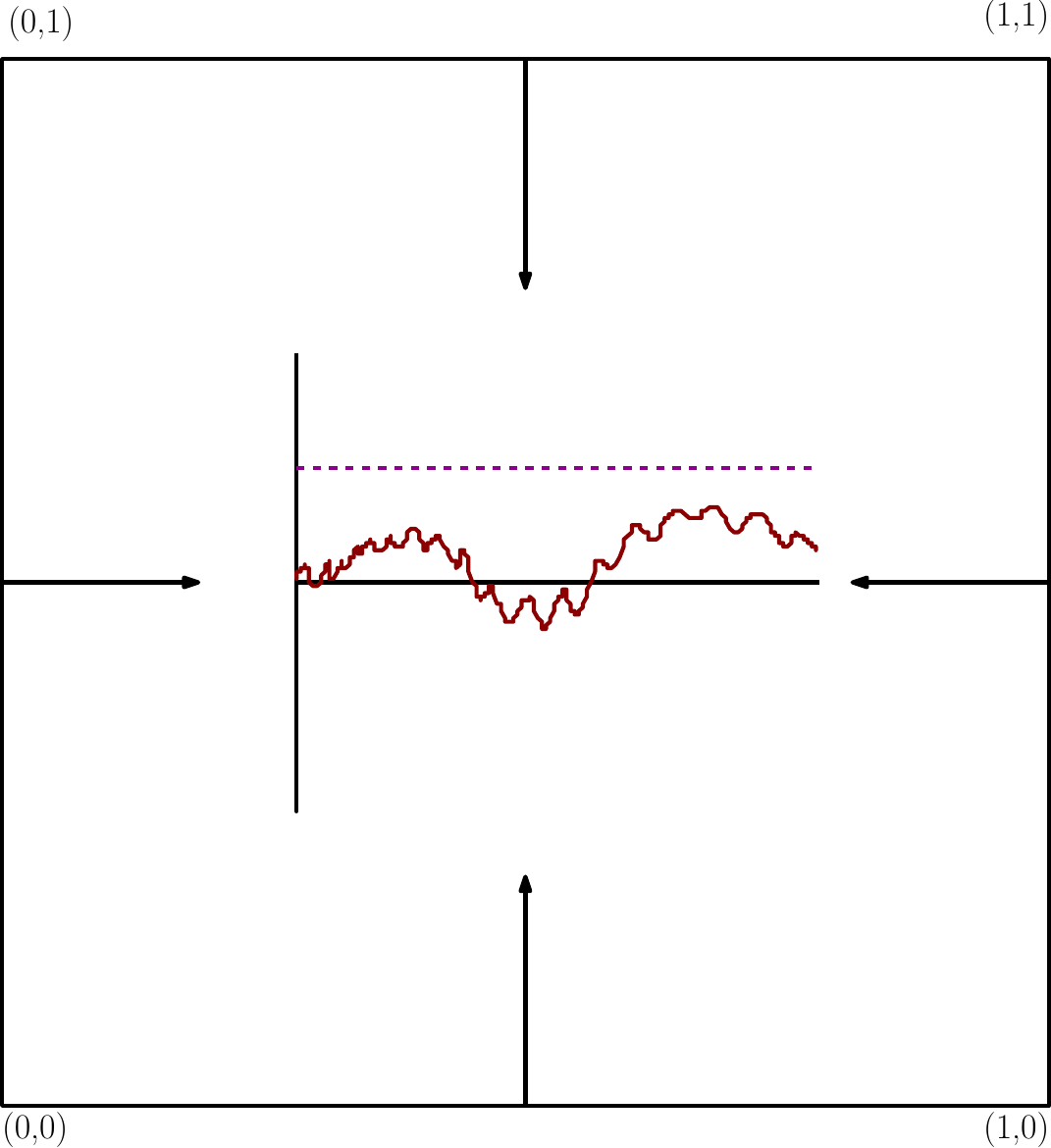}	
	\end{subfigure}
	\begin{subfigure}{0.4\textwidth}
		\includegraphics[width=0.9\linewidth]{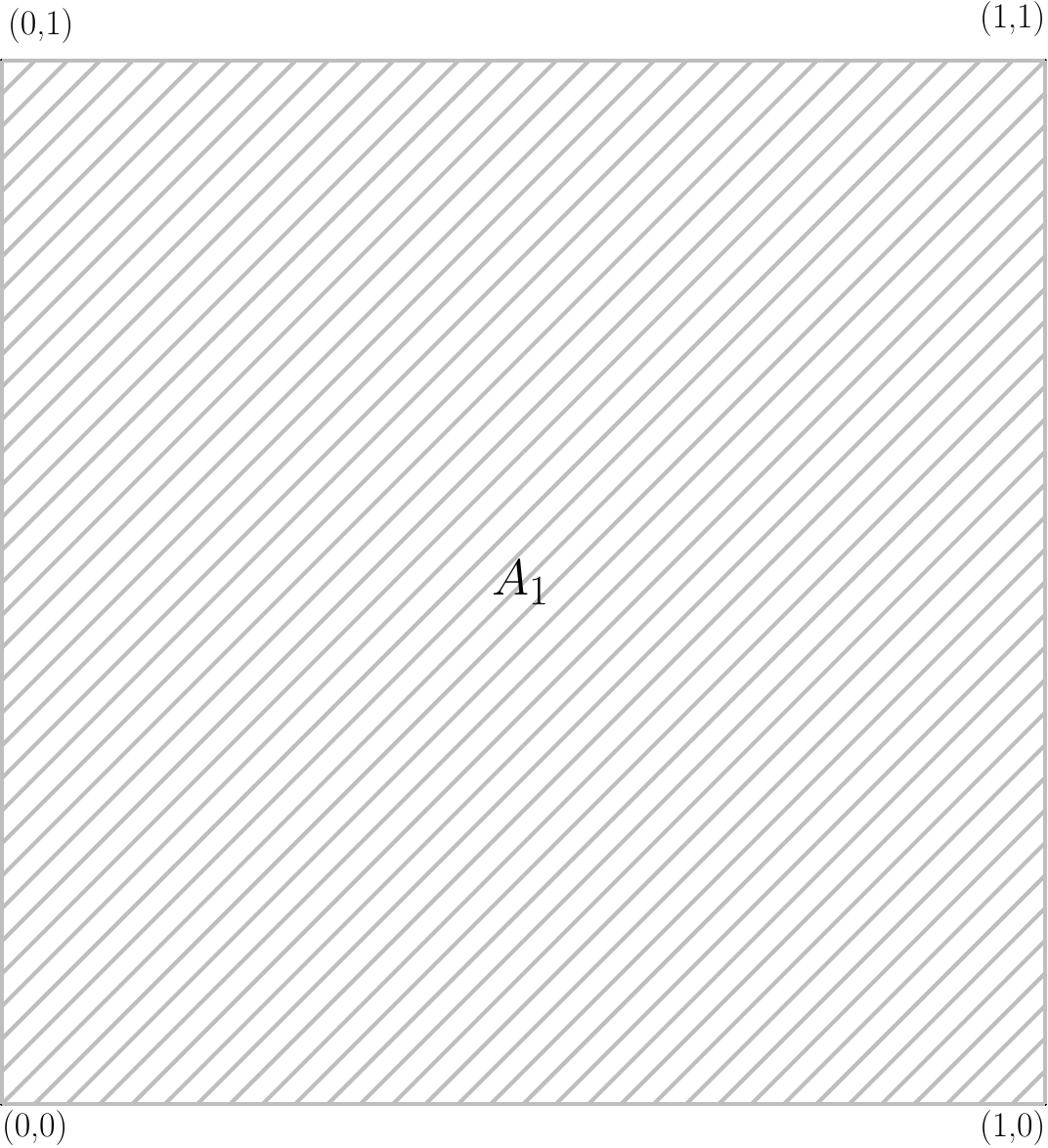}	
	\end{subfigure}
	\begin{subfigure}{0.4\textwidth}
		\includegraphics[width=0.9\linewidth ]{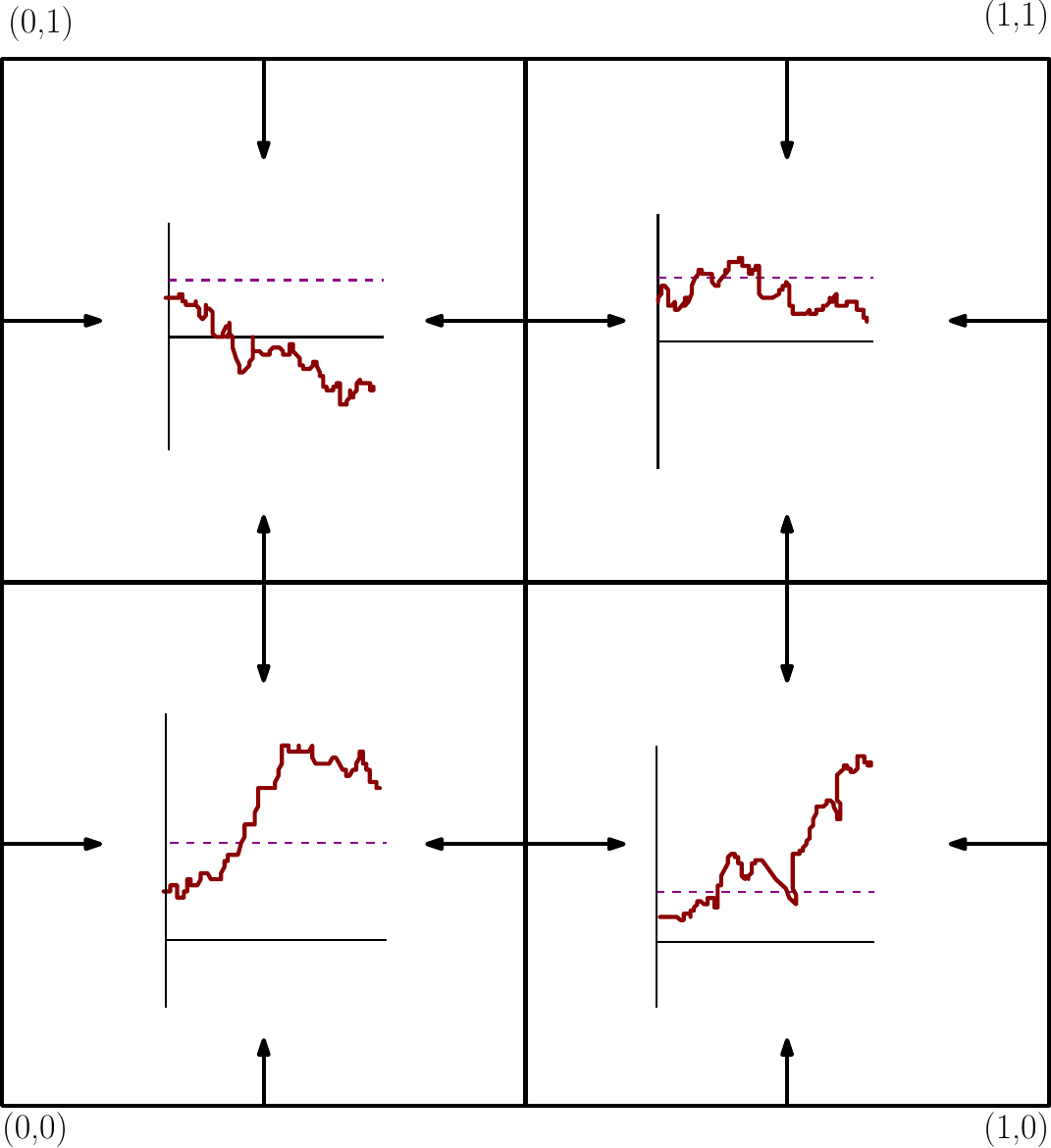}	
	\end{subfigure}
	\begin{subfigure}{0.4\textwidth}
		\includegraphics[width=0.9\linewidth]{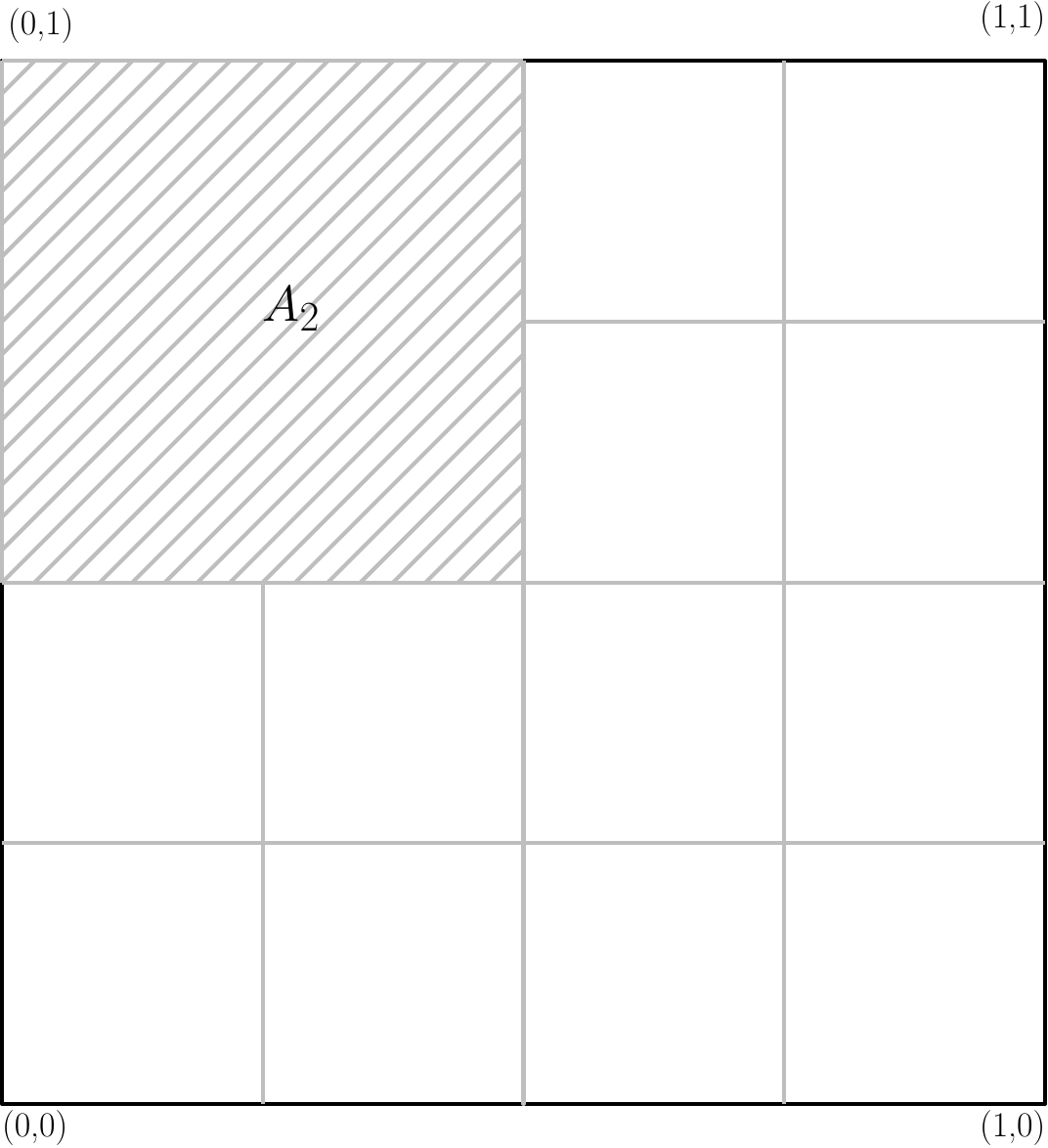}	
	\end{subfigure}	
	\caption{First two steps in the construction of $A$. In the left pictures we represent the Brownian motion associated to each point. In the right figure, the grey areas represents $A_n$.}
	\label{A}
\end{figure}

\medbreak
\noindent
{\bf Definition of $A$.}
If $B$ hits $1$ before time $T$, we define $A^1$ to be equal to the union of these four segments at the end-time $T$ of this exploration, so that $U_1 := S^\emptyset \setminus A^1 = S^1 \cup \ldots \cup S^4$. If not we take $A^1=D$. Note that $\E [ (\Gamma,1) | \sup_{t \le T} B_t  \ge 1]  = 1$. 

If the Brownian motion has not reached $1$ before time $T$, we continue exploring, and we do this independently and simultaneously in all four squares $S^1, \ldots, S^4$ 
using the GFF $\Gamma^{A^1}$ in each of them (note that $\Gamma^{A^1}$ consists of four independent GFFs in the four squares). 
In each of these squares, we grow four boundary segments towards the center of the square, and we study the conditional expectation of $4(\Gamma^{A^1},\1_{S^j})$ (the mean of the mass of $\Gamma^{A^1}$ in $S^j$)  given what one has discovered. By self-similarity, each of these 
four quantities evolve like four independent Brownian motions $B^1, \ldots , B^4$ up to time $T$. 

Now, in order to define $A^2$ we have two cases: if $A^1\neq D$, then $A^2= A^1$. If not, we look, for each $S^i$, at whether the BM $W^{i}:=B(t\wedge T) + B^i(t-T)\I{t\geq T}$ hits level $1$ before time $2T$ or not.  $A^2$ is made by the closed union of all the squares of size $2^{-1}$ where this $BM$ did not hit the level 1 before time $2T$, with the boundaries of all the squares of the same size where this event happen (see Figure \ref{A}). In other words, for each $n \ge 1$: 
\begin {itemize}
 \item The sets $A^n$ and $\partial A^n$ are local sets made out of the union of $2^{-n}$ dyadic segments with elements of $\SS_n$, and $A^n$ is such that $(A^n)_n=A^n$. We say that a square $s\in \SS_n$ is still active (meaning that we will continue exploring inside it) when $s\in A^n$.  Furthermore, active squares also come equipped with a Brownian motion $W^s$ stopped at time $Tn$. We call $K_n$ the set of active squares in $\SS_n$ and $V_n$ the set of  connected components of $D\backslash A^n$, i.e., the inactive components. Note that $V_n\subseteq \bigcup_{k=1}^n \SS_n$. 
\item 
In order to construct $A^{n+1}$ and to continue $W$, we proceed as follows:  The components that were not active at step $n$ remain inactive.  For $s\in K_n$, continuously grow the middle lines as done in the first step. Now, define for $0\leq t \leq n(T+1)$ and $s^+$ any direct descendent of $s$, $W^{s^+}(t):=W^{s}(t\wedge nT)+B^{s}(t-nT)\I{t	\geq nT}$ , where $B^{s}$ is the BM associated with the change of the conditional expectation of $2^{n}(\Gamma^{A^{n}},1_{s})$ given the increasing procedure in $s$. We keep active those squares $s^+$ where its associated BM did not hit $1$ before time $(n+1)T$, and we make $s^+$ inactive (i.e. $s^+\in V_{m}$ for $m\geq n+1$) if its associated BM hit $1$ before time $(n+1)T$. We define $A^{n+1}$ as the closed union of all the active squares at time $(n+1)$ with the boundary of the inactive squares. We can also see it as $A^n$ minus the squares $s^+$ that became inactive in this step.
\end {itemize}

Note that $A^n$ is non-increasing and that the family $V_n$ is non-decreasing. We define $A$ to be the intersection of all $A^n$. The complement of $A$ is then just the union of the squares that stop being active at some point, more precisely, $D\backslash A$ is the disjoint union of the squares in $\cup_n V_n$. Thus, we have that $A_n=A^n$. Note that for a given dyadic square $s$, on the even that $s \in V_n$,  the harmonic function $h_A$ coincides with the harmonic function $h_{D \setminus T_n}$ on $s$ (where 
$T_n$ the union of all boundaries of $2^{-n}$-dyadic squares) and that $(\Gamma_A,1_s) = \gamma_n (s)$. 

\medbreak
\noindent
{\bf The set $A$ is not large.} 
It follows from the construction that the probability that a given dyadic square $s$ of side-length $2^{-n}$ is still active at step $n$ is equal to the probability that a one-dimensional Brownian motion did not hit $1$ before time $n T$. This probability decays like a constant times $1/ \sqrt {n}$ as $n  \to \infty$. From this, it follows readily that the size of $A$ is indeed of the type required for \hyperlink{22}{(2)$_2$}.
\begin {prop}
\label {22}
The expected value of the area of the $\eps$-neighbourhood of $A$ decays almost surely like  $O(|\log \eps|^{-1/2})$.
\end {prop}
\begin{proof}
	Indeed, if $N_n = N_n (A)$ denotes the number of closed $2^{-n}$ dyadic squares that intersect $A$, then 
\begin{align*}
\E\left[N_n \right]&=\sum_{s\in \SS_n}\E\left[\I{s\subseteq A_n}\right ]+C\sum_{j=1}^{n-1}\sum_{s\in \SS_j}2^{n-j}\E\left[\I{s\subseteq A_j\backslash A_{j+1}}\right ]\\
&\leq 4^{n} \P\left (\text{BM does not hit $1$ before } Tn \right)+C2^{n}\sum_{j=1}^{n-1}j^{-3/2}2^{-j}
\leq C \frac{4^{n}}{\sqrt{n}}
\end{align*}
(mind that in $N_n$, we have also to count the squares that intersect the boundaries of squares that have stopped being active, which explains the sum in $j$). \end{proof}

\medbreak
\noindent
{\bf A first moment estimate.} 
Note that to define the set $A$, we have in fact associated a Branching Brownian motion (BBM) $W$ to each GFF, where each BM splits into $4$ independently evolving BM at each time which is a multiple of  $T$.  However, it should be emphasised that for a given dyadic square $s$ of side-length 
$2^{-n}$, the value of the corresponding Brownian motion at time $nT$ is not equal to the expected mean height of the GFF in $s$ given the exploration up to the $n$-th generation. Indeed, this mean height has a higher value when $s$ is towards the centre of $D$ than when it is near its boundary. This phenomenon is not mirrored by  the Branching Brownian motion description. However, a key observation is that this difference is averaged out when summing over all squares. For instance, it is easy to check by induction on $n$ that
\[ \sum_{s \in {\mathcal S}_n} \gamma_n (s)  = \sum_{s  \in {\mathcal S}_n} 4^{-n} W^{s} (nT),\]
where $W^s$ denotes the Brownian motion that follows the branch of the BBM corresponding to $s$. 

The variant of this result that is useful for us is the following.
\begin {lema}
\label {l1}
 $$  \E\left [  (\Gamma_{A},\1_{D\backslash A_n})  \right ] = \E \left [ \sum_{s \in V_n} \hbox {Area}(s) \right ].$$ 
\end {lema}
The right-hand side is equal to the probability that a Brownian motion started from $0$ hits $1$ before time $nT$, which converges to $1$. 
This shows already that  $(\Gamma_{A},\1_{D\backslash A_n})$ can not converge in $L^1$ to $(\Gamma_A, 1)$, which is a symmetric random variable with mean $0$. 

\begin {proof}
Note that $D\backslash A_n=\bigcup_{s\in V_n} s$ and that at time $n$, $\Gamma_{\partial A_n}=\Gamma_A$ in all elements of $V_n$ and $\Gamma_{\partial A_n}=\Gamma_{T_n}$ in all of those in $K_n$. This implies that $\E\left [  (\Gamma_{A},\1_{D\backslash A_n})  \right ]=-\E\left[\sum_{s\in K_n} \gamma_n(s) \right]$. Then, it is enough to prove that 
\begin{align*}
\E\left [ \sum_{s\in K_n} \gamma_n(s)  \right ]= 4^{-n}\E\left[ \sum_{s\in K_n} W^s \right]=-\E\left[\sum_{s\in V_n} Area(s) \right].
\end{align*}

The second equality just follows from the optional stopping  theorem. For the first equality we have to work harder. Take $m\leq n\in \N$ and fix $s'\in \SS_m$, we have that for all  $s \in \SS_n$ with ancestor $s'$,  $W^{s}((m+1)T)-W^s(mT)$ is equal to $4^{m}(\gamma_{m+1}(s')-\gamma_{m}(s'))$ and that $\E\left[ \I{s\in K_n }\mid \Gamma_{T_{m+1}}\right]$ does not depend on $s$. Now, let us show that the increment of the harmonic function for $s\in K_n$ at level $m$ can be computed using the Brownian motion,
\begin{align*}
\sum_{\substack{s\in \SS_n:s' \leq s}}\E\left[(\gamma_{m+1}-\gamma_{m})(s)\I{s\in K_n }\right]&=\sum_{\substack{s\in \SS_n:s' \leq s}}\E\left[(\gamma_{m+1}-\gamma_{m})(s)\E\left[ \I{s\in K_n }\mid \Gamma_{T_{m+1}}\right]\right]\\
&=4^{m-n} \E\left[(\gamma_{m+1}-\gamma_{m})(s')\E\left[ \I{s\in K_n }\mid \Gamma_{T_{m+1}}\right]\right]\\
&=4^{-n}\sum_{\substack{s\in \SS_n:s' \leq s}} \E\left[(W^{s}((m+1)T)-W^s(mT))\I{s\in K_n } \right] .
\end{align*}
We conclude by writing a $\sum_{s\in K_n}\gamma_n(s)$ as a telescopic sum.
\end{proof}

\medbreak 
\noindent
{\bf This set $A$ is not thin.}
Our goal is now to derive the following fact, which combined with Proposition \ref {22} proves the statement \hyperlink{22}{(2)$_2$}:  
\begin {prop}
The local set $A$ is not thin. 
\end {prop}
This is a direct consequence of the following claim:
\begin {claim}
\label {l3}
The sequence of random variables $(\Gamma_A,\1_{D\backslash A_n})$ is bounded in $L^2$.
\end {claim}
Indeed, if $ (\Gamma_A,\1_{D\backslash A_n})$ would converge in probability towards $(\Gamma , 1)$, then it would converge also in $L^1$, and we have seen in the previous paragraph that 
this can not be the case. 

Deriving Claim \ref {l3}  requires some care. We have to bound covariances of the increments of the integral of the harmonic function in two squares, $s$ and $s'$, at each step of the process. To do that, we separate the increments according to whether or not they come from the conditional expected value of $T_m$ with $m$ bigger or equal, $p$, the height of $s\wedge s'$, the last common ancestor of $s$ and $s'$. We realise that if we condition according to the values of the GFF in $T_p$  many terms become constant and allow us to go the increments of level $p$, instead of $n$.

\begin{proof}[Proof of the claim]
As in the beginning of Lemma \ref{l1}, $(\Gamma_A,\1_{D\backslash A_n})=(\Gamma_A,1)+\sum_{s\in K_n}\gamma_n(s)$. Given that $\Var(\Gamma_A,1)\leq \Var (\Gamma,1)$ it is just enough to bound  
$$ 
\E \left [ \sum_{s, s' \in K_n} \gamma_n (s) \gamma_n(s') \right ].$$
We do this by writing $\gamma_n (s)$ and $\gamma_n (s')$ as the sum of the increments at each iteration step. Things are a little bit messier than for the first moment because one has more terms to evaluate.  For $s,s' \in \SS_n$, we will have to consider the common ancestor  $w=s\wedge s'$. In the following lines, we first fix $p \ge 2$ and $w$ a $2^{-p}$-daydic square.  

 For any $m,o\geq p$  conditionally on $\Gamma_{T_p}$,  $(\gamma_{m+1}-\gamma_{m})(s) \I{s\in K_n}$ and  $(\gamma_{o+1}-\gamma_{o})(s') \I{s\in K_n}$ are independent. 
 Hence,
\begin{align*}
&\sum_{p\leq m,o< n}\sum_{\substack{s,s'\in \SS^n\\s\wedge s'=w}} \E\left[(\gamma_{m+1}-\gamma_m)(s)(\gamma_{o+1}-\gamma_o)(s')\I{s,s'\in K_n}\right] \\
=&\sum_{p\leq m,o<n}\sum_{\substack{s,s'\in \SS^n\\s\wedge s'=w}}\E\left[\E\left[ (\gamma_{m+1}-\gamma_m)(s)\I{s\in K_n}\mid \Gamma_{T_p}\right]\E\left[ (\gamma_{o+1}-\gamma_o)(s')\I{s'\in K_n}\mid \Gamma_{T_p}\right]\right] \\
=& \sum_{\substack{s,s'\in \SS^n\\s\wedge s'=w}}8^{-n}\E\left[ \E\left[ (W^s(nT)-W^s(pT))\I{s\in K_n}\mid \Gamma_{T_p}\right]\E\left[ (W^{s'}(nt)-W^{s'}(pt))\I{s'\in K_n}\mid \Gamma_{T_p}\right] \right] \\
\leq & \sum_{\substack{s,s'\in \SS^n\\s\wedge s'=w}}8^{-n}\E\left[(W^s(pT)+1)^2\I{w\in K_p} \right] \leq C 8^{-p} \sqrt{p},
\end{align*}
where for the third equality we used the same technique as in Lemma \ref{l1} and for the fourth and fifth we just use the optional stopping theorem for the BM $B$ and for $B^2_t -t$.

It is also true that  $\P(u\in K_n\mid T_p)$ is constant for all $u$ with ancestor $w$ and that conditionally on  $\Gamma_{T_p}$, $\{s\in K_n\}$ is independent of $\{s'\in K_n\}$.
This allows us to compute the following second term
\begin{align*}
&\sum_{0\leq m,o<p} \sum_{\substack{s,s' \in \SS^n\\ s\wedge s=w}}\E\left[(\gamma_{m+1}-\gamma_m)(s)(\gamma_{o+1}-\gamma_o)(s')\I{s,s'\in K_n}\right]\\
=& \sum_{0\leq m,o<p} \sum_{\substack{s,s' \in \SS^n\\ s\wedge s=w}}\E\left[(\gamma_{m+1}-\gamma_m)(s)\P\left[\I{s\in K_n} \mid\Gamma_{T_p} \right] (\gamma_{o+1}-\gamma_o)(s')\P\left[\I{s'\in K_n} \mid\Gamma_{T_p} \right] \right]\\
=& \E\left[\gamma_p(w)^2 \I{w\in K_p} \right] \leq C 8^{-p}\sqrt{p} \log(p),
\end{align*}
where in the last step we have used \eqref{GB} and the fact that the variance of $\gamma_p(w)$ is bounded by that of $(\Gamma,\1_w)$.

For the remaining term we need to bound the cross -product and using similar remarks as before we have that
 \begin{align*}
&\sum_{0\leq m<p\leq o <n} \sum_{\substack{s,s' \in \SS^n\\ s\wedge s=w}}\E\left[(\gamma_{m+1}-\gamma_m)(s)(\gamma_{o+1}-\gamma_o)(s')\I{s,s'\in K_n}\right]\\
=& \sum_{0\leq m<p\leq o <n} \sum_{\substack{s,s' \in \SS^n\\ s\wedge s=w}}\E\left[(\gamma_{m+1}-\gamma_m)(s)\P\left[\I{s\in K_n} \mid\Gamma_{T_p} \right] \E\left[(\gamma_{o+1}-\gamma_o)(s')\I{s'\in K_n} \mid\Gamma_{T_p} \right] \right]\\
=& -\E\left[\gamma_p(w)(-W^{w}(pT)+1)c(W^{w}(pT),n-p) \I{w\in K_p}  \right] \leq C 8^{-p}\sqrt{p} \log(p),
\end{align*}
where $c(x,m)$ is the probability than a BM hits height $x+1$ before time $mT$.

Summing all the previous terms up, we get that 
\begin{align*}
\E\left[ \sum_{s,s' \in K_n}\gamma_n(s)\gamma_n(s')\right]\leq C' +  C\sum_{p=2}^\infty 4^{-p} \sqrt{p} \log(p)<\infty.
\end{align*}
\end {proof}

\subsection{The example in higher dimensions}

We now explain how to adapt the previous example to the higher-dimensional setting. 
The only slight is that in the two-dimensional case, we used the scale invariance of the GFF, while 
we will now  use the scaling relation \eqref{SGF}. 

To adapt our example, let us define $D =S^\emptyset:= (0,1)^d$. We use the $d$-dimensional dyadic hypercubes denoted now by $S^u$ where $u$ are finite sequences in $\{ 1, \ldots,2^d \}$. When $\Gamma$ is a GFF in $D$, we are now going to discover its values on all simultaneously growing all the  $(d-1)$-dimensional mid-hyperplanes. Then, the iterative construction proceeds in almost the same way, but with a notable difference. Due to the different scaling behaviour of the GFF, if the evolution of the conditional mean height during the first iteration evolves like a Brownian motion up to some time $T$, then the evolution during the second iteration is that of a Brownian motion during time $T \times 2^{d-2}$, and so on. In other words, the intervals between the branching times of the branching Brownian motion will grow exponentially, and the $n$-th branching time will be $T_n = T  (2^{(d-2)n} -1 ) / ( 2^{d-2} - 1 )$ instead of $nT$. 

Other than that, nothing in the previous discussion changes. Lemma \ref {l1} together with Claim \ref{l3} become readily: 
\begin {lema}\label {l2}  For this $A$ we have that $  \E\left [  (\Gamma_{A},\1_{D\backslash A_n})  \right ] = \E \left [ \sum_{s \in V_n} \hbox {Volume} (s) \right ]$ and the second moment of $(\Gamma_{A},\1_{D\backslash A_n})$ is uniformly bounded. 
\end {lema}
Just as in the 2-dimensional case, this then implies that $A$ is not thin.

To upper bound the Minkowski dimension, the only difference is that the probability that a given dyadic hypercube of side-length $2^{-n}$ is active at the $n$-th iteration is now the probability that a Brownian motion does not hit level $1$ before time $T_n$, which leads to the estimate on the size of $A$ as in \hyperlink{2d}{(2)$_d$}. Indeed, if $N_n$ denotes the number of closed dyadic hypercubes that intersect $A$,
\begin{align*}
\E\left[ N_n \right]&\leq C \sum_{j=1}^{n}\sum_{s\in \SS^j}2^{(n-j)(d-1)}\E\left[\I{s	\subseteq A_j} \right]=C2^{n(d-1)}\sum_{j=1}^n 2^{j}\P\left (\text{BM hits 1 after time } T_n\right ) \\
&\leq C 2^{n(d-1)}\sum_{j=1}^n  2^{\left (-d/2+2\right )j}\leq C 2^{\max\{d-1,d/2+1\}n} .
\end{align*}
Thus, thanks the Markov inequality
\begin{align*}
\P\left[N_n \geq 2^{(\max\{d-1,d/2+1\}+\epsilon)n} \right] \leq C2^{-\epsilon n},
\end{align*}
and thanks to the Borel-Cantelli Lemma, we can conclude that the upper Minkowski dimension of $A$ is almost surely bounded by $\max\{d-1,d/2+1 \}$.

We conclude that \hyperlink{2d}{(2)$_d$} holds for any $d \ge 3$.
\begin {prop}[(2)$_d$]
 This local set $A$ is not thin, and its upper Minkowski dimension is almost surely not larger than $\max\{d-1,d/2+1 \}$.
\end {prop}

\section{Small sets are thin (proof of (1)$_d$)} \label{small thin d}

Let us briefly note that the definition of thin sets can be extended to non-local sets: we say that a set $A$ is thin if for all $f$ smooth bounded function in $D$ we have that  $(\Gamma,f\1_{A_n})\stackrel{\P}{\to}0$ as $n\to \infty$. This definition is useful because  a.s. 
\begin{equation}\label{eq::thin}
\sum_{s\in \SS_n:s\nsubseteq D\backslash A_n}(\Gamma,f\1_s)=(\Gamma,f\1_{A_n}),
\end{equation}so that it is sufficient to bound the value of the GFF in hyper-cubes of size $2^{-n}$.

The following proposition links both definitions.
\begin{lema} \label{BE}
	Let $\Gamma$ be a GFF on $D$ and $A$ a local set. $A$ is thin in this last sense if only if $A$ is a thin local set.
\end{lema}
\begin{proof}
	It is enough to see that for all $f$ smooth and bounded function:
	\begin{align*}
	(\Gamma,f\1_{A_n})-((\Gamma_A,f)-(\Gamma_A,f\1_{D\backslash A_n})) = (\Gamma^A, f\1_{A_n})\stackrel{\P}{\to} 0 \ \ \ \ \ \text{ as } n \to \infty,
	\end{align*}
\end{proof}
This shows for instance that any deterministic closed set $A$ with zero Lebesgue measure is a thin local set. Indeed, if $\|f\|_\infty<1$, by dominated convergence, 
\begin{align*}
\E\left [ (\Gamma,f\1_{A_n})^2\right] = \iint_{A_n\times A_n} f(x) G_D(x,y)f(y)\d y\d x  \to 0
\end{align*}
as $n \to \infty$. 

\subsection {The case $d \ge 3$}
Now, we want to show that for any set with Minkowski dimension smaller than $1+(d/2)$ satisfies \eqref{eq::thin}. To do this, let us see how big are the values which actually ``give mass'' to the GFF.
\begin{lema}\label{lema::mass}
	Let $d\geq 3$ and $\Gamma$ be a GFF in $D\subseteq \R^d$. Then, there exists a deterministic constant $C_d$ such that for any bounded function $f$ with $\|f\|_\infty\leq 1$,
	\begin{equation}\label{eq::size}
	\E\left[\sum_{s\in \SS_n}|(\Gamma,f\1_{s})|\I{|(\Gamma,f\1_{s})|\geq C_d\sqrt{n}2^{-(d/2+1)n}}\right] \to 0 
	\end{equation}
	where $C_d$ is a deterministic constant.
\end{lema} 
\begin{proof}
To begin, let us recall that there exists an absolute constant $\bar C_d$ such that for any $s\in \SS_n$ and any bounded $\|f\|_\infty\leq 1$,
\begin{align}\label{BVd}
\iint_{s\times s} f(x)G(x,y)f(y)dxdy\leq \bar C_d  2^{-(d+2)n}.
\end{align}
By an exact computation we have that, if we define $K_d:=C_d^2/(2\bar C_d^2)$, then
\begin{linenomath}
	\begin{align*}
	\sum_{s\in \SS_n}\E\left[|(\Gamma,f\1_{s})|\I{|(\Gamma,f\1_{s})|\geq C_d\sqrt{n}2^{-(d/2+1)n}}\right]\leq 2^{nd}2^{-(d/2+1)n}e^{-K_dn}\to 0,
	\end{align*}
\end{linenomath}
when $K_d> \log(2)\cdot (d/2-1)$.
\end{proof}

We can now use the lemma to prove \hyperlink{1d}{(1)$_d$}.
\begin{prop}[\hyperlink{1d}{(1)$_d$}] \label{Thin Sets prop}
	Let $D\subseteq \R^d$ be an open set, $\Gamma$ a GFF in $D$ and $A$ a local set of $\Gamma$. If the upper Minkowski dimension of $A$
	is almost surely strictly smaller than $\max\{d-2,d/2+1\}$,  then $A$ is a thin local set.
\end{prop}

\begin{proof} 
	Let us first note that if the upper Minkowski dimension $\delta (A)$ of $A$ is strictly smaller than $d-2$, then $A$ is polar, so that
	Lemma \ref{polar} implies that $\Gamma_A=0$, and thus $A$ is a thin local set.

	The following argument will in fact not use the fact that $A$ is a local set. Note that WLOG we can take $\|f\|_\infty\leq 1$. Let us now define $M_n:=M_n(A)$ as the amount of open dyadic squares of size $2^{-n}$ that intersect $A$. Then, by studying whether the integral of the field on each square is smaller than $ C_d\sqrt{n}2^{-(d/2+1)n}$, we have that $\P\left (|(\Gamma,f\1_{A_n})|\geq \epsilon\right )$ is smaller than or equal to
	\begin{align}\label{e. separation}
 \P\left( \sum_{s\in \SS_n}|(\Gamma,f\1_{s})|\I{|(\Gamma,f\1_{s})|\geq C_d\sqrt{n}2^{-(d/2+1)n}}\geq \epsilon \right)+ 	\P\left (M_n 2^{-(d/2+1)n}\sqrt{n}C_d\geq \epsilon\right ) .
	\end{align}
	The first term converges to $0$ as $n\to \infty$ thanks to Lemma \ref{lema::mass}. Also, as $n\to \infty$ the second term converges to $0$. To see this, note that $M_n\leq N_n$, the amount of closed dyadic squares that intersect $A$. This implies that the second term is smaller than or equal to $\P(N_n\geq \epsilon C_d2^{(d/2+1)n}\sqrt{n})$. This term converges to 0 because the Minkowski dimension of $A$ is smaller than $1+(d/2)$.
\end{proof}

\begin{rem}
	Let us note that the proof of Proposition \ref{Thin Sets prop} can be improved in the case where $d$ is either $3$, $4$ or $5$. In this case, if $N_n(A) 2^{-(d/2+1)n}\sqrt{n}C_d$ converges to $0$ in probability, then $A$ is thin.
	\end{rem}

Note that with this proposition and its proof we can get some other basic properties of thin sets.
\begin{cor} \label{Cord}Let $D\subseteq \R^d$ be an open set, $\Gamma$ a GFF on $D$ and $A$, $B$ thin local sets. If the upper Minkowski dimension of $A$
is strictly smaller than $d/2+1$, then:
	\begin{enumerate}
		\item $A\cup B$ is also a thin local set.
		\item If $h_A$ is integrable (i.e., such that  $\int_{D \setminus A} | h_A | < \infty$) and $B$ has 
		zero Lebesgue measure, then a.s. $B\backslash A$ is thin for $\Gamma^A:=\Gamma-\Gamma_A$.
	\end{enumerate}
\end{cor}	

\begin{proof}
	\begin{enumerate}
		\item 	Note that for any bounded smooth function $f$: 
		\begin{align*}
			|(\Gamma,\1_{(A \cup B)_n})|
			&\leq |(\Gamma,f\1_{B_n})|+|(\Gamma,f\1_{A_n\backslash B_n)})|\stackrel{\P}{\to} 0, \ \ \ \ \text{as } n\to \infty,
		\end{align*}
		where the second term goes to $0$ because it can be written as a sum over elements of $\SS_n$ and the amount of terms in that sum is smaller than the cardinal $\{s\in \SS_n: s\subseteq A_n  \}$. Thus, one can bound the probability of it being bigger than $\epsilon>0$ by the analogue of \eqref{e. separation}. The same argument used in the proof of Proposition \ref{Thin Sets prop} shows the convergence to $0$.
		\item Let $f$ be a bounded function and note that the fact that because $h_A$ is integrable and $B$ has 0 measure $\int_{D\backslash A} h_A(x) \1_{(B\backslash A)_n} f(x) \d x$ goes to 0. Additionally $(\Gamma,f \1_{(B\backslash A)_n} )$ because of the same reason as in the proof of (1).
	\end{enumerate}
\end{proof}

In future work, we plan to prove that when the upper Minkowski dimension of $A$ is smaller than $d/2 +1$, then   $h_A$ is  integrable on $D\backslash A$, which will allow to relax a little bit the conditions in this last corollary. 

 Note that this does not answer the question whether the fact that $B$ is thin implies that its Lebesgue measure is $0$. 
Remark that such statements are non-trivial, due for instance to the fact that we cannot exclude at this point, the fact that there exist thin local sets, with 
non-thin local subsets.

\subsection{The case $d=2$}\label{small thin 2}
This case is similar to general dimension, so we just remark where the differences lie.

We need a lemma analogue to Lemma \ref{lema::mass}.
\begin{lema}\label{lema::mass2}
	Let $d\geq 2$ and $\Gamma$ be a GFF in a bounded domain $D\subseteq \R^2$. Then for any bounded function $f$ such that $\|f\|_\infty\leq 1$,
	\begin{equation}\label{eq::size2}
	\E\left[\sum_{s\in \SS_n}|(\Gamma,f\1_{s})|\I{|(\Gamma,f\1_{s})|\geq C\sqrt{\log(n)}\sqrt{n}2^{-2n}}\right] \to 0 
	\end{equation}
	where $C$ is a deterministic constant.
\end{lema}
\begin{proof}
	We just need to note that for any bounded $D\subseteq \R^2$, there exists a constant $\bar C$ such that:
	\begin{align}\label{V2}
	\iint_{s\times s} f(x)G(x,y)f(y)dxdy\leq \bar C  2^{-4n}\log n.
	\end{align}
	Thus, similarly to Lemma \ref{lema::mass}
	\begin{linenomath}
		\begin{align*}
		\sum_{s\in \SS_n}\E\left[|(\Gamma,f\1_{s})|\I{|(\Gamma,f\1_{s})|\geq C_d\sqrt{n}2^{-(d/2+1)n}}\right]\leq 2^{2n}\sqrt{n}2^{-2n}e^{-K\log(n)}\to 0,
		\end{align*}
	\end{linenomath}
	when $K> 1/2$.
\end{proof}

Now,  we can conclude \hyperlink{12}{$(1)_2$} using the same reasoning as Proposition \ref{Thin Sets prop}
\begin{prop}[\hyperlink{12}{(1)$_2$}] \label{Thin Sets prop 2d}
	Let $D\subseteq \mathbb C$ be a bounded open set, $\Gamma$ a GFF on $D$ and $A$ local set. If there exists $\delta>0$ such that 
	$$  \E\left[N_n\right] = o \left( \frac{4^{n}}{\sqrt{n\log(n)}}\right)  $$
	as $n\to \infty$, then $A$ is a thin local set.
\end{prop}
\begin{rem}
	Let us note that there is also an equivalent of Corollary \ref{Cord} when $d=2$. In this case we just need that $A$ and $B$ are thin local sets and that $\E\left[N_n(A)\right]$ is $o(4^n/\sqrt{n\log(n)} )$. 
\end{rem}

\section{Some comments about the definitions of thin local sets}\label{Thin sets def}

Let us now make some somewhat abstract comments about the definition of local sets. 
One general strategy used to define local sets is to use some deterministic ``enlargements'' of the random sets $A$ (see for instance \cite{WWln2}). To the best of our knowledge, only dyadic-type enlargements have been used in earlier works, but this is a somewhat arbitrary choice. 
For our purposes here, it seems natural to consider also other possible deterministic enlargements -- indeed, this a priori choice could be important, given that some property may hold for one approximation scheme, and not for the other.

Let us describe one possible class of discrete approximation schemes (DAS), for which the proofs of the present paper can be adapted rather directly. 

\medbreak
\noindent \textbf{Discrete aproximation schemes.}
Define a  pre-DAS for a domain $D\subseteq \R^d$ to be a sequence $(\AA_n)_{n \ge 0}$ of families of closed sets $\AA_n=(\BB_n,\CC_n)$ for which there exists some (large) constant $C\in \R$ such that the following holds for any $n\in \N$: 
\begin{enumerate}
	\item For any two distinct $c$ and $c'$ in $\CC_n$, the Lebesgue measure of $c \cap c'$ is zero. 
	\item For any $c$ in $\CC_n$ the diameter of $c$ is upper bounded by $C 2^{-n}$ and its volume is lower bounded by $2^{-nd}/C$. 
	\item 	$\text{Leb}(\bigcup_{b\in \BB_n} b) = 0$. And for all $E\subseteq \R^d$ compact, the cardinal of the elements of $\BB_n$ that intersect $E$ is finite
\end{enumerate} 

For a fixed pre-DAS $\AA_n$, take $\BB^n:=\bigcup_{b\in \BB_n} b$, the set of all points covered by elements of $\BB$. For all closed set $A\subseteq \bar D$,  define $\AA\{A\}_n$ as the set of all elements of $\CC_n$ that have a non empty intersection with $A\backslash \BB^n$ and take $\AA[A]_n$ the union of all sets in $\AA\{A\}_n$ with all the set in $\BB_n$ that have non-empty intersection with $A$. More formally, 
\begin{align*}
&\AA\{A\}_n:=\{c\in \CC_n: c\cap A\backslash \BB^n\neq \emptyset \},\\
&\AA[A]_n:=\bigcup_{\substack{c\in\{A\}_n} }c\cup \bigcup_{\substack{b\in \BB_n,\\b\cap A \neq \emptyset}} b.
\end{align*}
We then say that a pre-DAS $\AA_n$ is a DAS if for all closed set $A\subseteq \bar D$, $\AA[A]_n\searrow A$. 

In this context, we understand $\AA[A]_n$ as an approximation of $A$ using a union of elements in $\BB_n$ and $\CC_n$. It should be understood that 
the elements of $\CC_n$ are the only ones ``giving mass'' to $\AA[A]_n$. $\AA\{A\}_n$ represents all the set in $\CC_n$ that where used to construct $\AA[A]_n$.

Dyadic hyper-cubes provide an example of DAS -- more precisely, when $\CC_n$ are the closed dyadic hypercubes of side-length $2^{-n}$ intersected with $D$ and $\BB_n$ is empty. 
This is our canonical DAS and it is such that for all closed sets $A$ the cardinal of $\AA\{A\}_n$ is $N_n$.

Let us remark that condition (2) implies that if $A$ is bounded   $|\AA\{A\}_n|\leq CN_n$ and  that there exists an absolute constant $C_d$ such that for any $c\in \CC_n$ there exists $C_d$ such that \eqref{BVd} or \eqref{V2} holds.

\medbreak
\noindent \textbf{The generalised thin local sets.}
We are now ready to give an alternative definition of thin local sets. This definition coincides with that of \cite{WWln2} in the particular case when $h_A$ is integrable on $D \setminus A$ (so that working with DAS is not necessary). It is also similar to Lemma 3.10 \cite{MS1}, where they ask $\Gamma$ to be a.s. determined by the restriction of $\Gamma$ to $D\backslash A$. On the other hand, the first example presented in Section \ref{hints} is non-thin, but it is proven in \cite{ALS} that $\Gamma$ is a function of the restriction of $\Gamma$ to $D\backslash A$. This is because $A$ is measurable of this restriction and $\Gamma_A$ is a measurable function of $A$.
	
\begin{defi}[Generalised thin local sets] Let $\Gamma$ be a Gaussian free field on a domain $D$ and $A\subseteq D$ a local set. 
	We say that $A$ is a generalised thin local sets if for all $f$ smooth and with bounded support in $\R^d$ ($C_0^\infty(\R^d)$) and for all DAS $\AA_n$, the sequence $(\Gamma_A,f\1_{D\backslash\AA[A]_n})$ converges in 
	probability to $
	(\Gamma_A,f)$ when $n \to \infty$. 
\end{defi}

	Note that  $(\Gamma_A,f\1_{D\backslash\AA[A]_n})$ is always well defined thanks to the fact that when the supp$(f)$ is compact, $\AA[A]_n$ can take only finitely many values. 
	Also, as we have said before, if  $\int_{D \setminus A} |h_A| < \infty$, then the limit of $(\Gamma_A,f\1_{D\backslash\AA[A]_n})$ is a.s. equal to $\int_{D\backslash A} f(z)h_A(z)$ and this limit does not depend on the chosen DAS. The DAS framework is relevant in the case where the integral of $|h_A|$ on $D \setminus A$ diverges. 
	
	Additionally, when $\int_{D\setminus A} |h_A| < \infty $ it is actually enough to check the criteria for functions $f$ in $C_0^\infty(D)$, because when we approximate one function in $C_0^\infty(\R^d)$ restricted to $D$ by one in $C_0^\infty(D)$ both the left and right term of the definition converge to what they should.

Let us briefly note that the definition of thin sets can be extended to non-local sets: We say that a set $A$ is thin if for all $f\in C_0^\infty(\R^d)$ and for all DAS $\AA_\cdot$
\begin{align*}
\sum_{c\in \AA\{A\}_n}(\Gamma,f\1_{c})=(\Gamma,f\1_{\AA[A]_n})\stackrel{\P}{\to}0, \ \ \ \ \ \ \ \text{as } n\to \infty.
\end{align*} 
it is easy to see that an analogue of Lemma \ref{BE} also holds in this setup. This, together with the estimates \eqref{BVd} and \eqref{V2} allow us to prove two facts:
\begin{itemize}
	\item When a deterministic set $A$ has 0 Lebesgue measure it is generalised thin.
	\item  If the hypothesis of Proposition \ref{Thin Sets prop} or \ref{Thin Sets prop 2d} hold, then the set $A$ is generalised thin.
\end{itemize}

Additionally, note that when a set is a generalised thin local set, then it is thin. This implies that the sets $A$ defined in Section \ref{hints} are not thin local sets.

Let us remark that there are non-local random sets that are thin for one approximation scheme but not for another. Because in this paper we are mostly interested in local sets, we will only sketch the proof of this fact for $d=4$. Let us see that one can find non-thin sets that live in a deterministic hyperplane of dimension $3$. This implies that the given DAS actually does matter as one can always use a $\BB_n$ whose union contains this hyperplane. We construct this set by iteratively dividing a fix hyperplane using dyadic hypercubes, and only keeping those on which the integral of the GFF is bigger than $\delta 2^{-(d/2-1)n}=\delta 2^{-n}$. If we call this constructed set $A$, we can use a first and second moment estimate to show that if $\delta$ is small enough, there is a probability bounded away from $0$ that $N_n(A)\geq 3^{-n}2^{3n}$. It is now clear that the set $A$ is not thin when using only the dyadic hypercubes, as when we are on the event $\{N_n(A)\geq 3^{-n}2^{3n}\}$:
\[\sum_{c\in \AA\{A\}_n}(\Gamma,f\1_{c})\geq \delta 2^{-n}N_n(A)\geq \left( \frac{4}{3} \right)^n\to \infty.\]

We will finish by stating two open question. We still don't know whether the fact that a local set is thin for one particular approximation scheme implies it is a generalised thin local set. In particular, we are not able to show that the set $A$ defined in Section \ref{hints} is not thin for the approximation using dyadic hyper-cubes. Another open question is whether the union two (generalised) thin local sets is always a (generalised) thin local set. On the other hand, it is true that for all the important applications that have arisen at this point in time, this fact has only been needed when at least one of the sets is small enough. 

\bibliography{biblio}

\begin{thebibliography}{HMP10}

\bibitem[ALS17]{ALS}
Juhan Aru, Titus Lupu, and Avelio Sep{\'u}lveda.
\newblock First passage sets of the 2d continuum {G}aussian free field.
\newblock {\em arXiv preprint arXiv:1706.07737}, 2017.

\bibitem[ALS18]{ALS2}
Juhan Aru, Titus Lupu, and Avelio Sep{\'u}lveda.
\newblock First passage sets of the 2d continuum {G}aussian free field.
\newblock {\em arXiv preprint arXiv:1805.09204}, 2018.

\bibitem[APS17]{APS}
Juhan Aru, Ellen Powell, and Avelio Sep{\'u}lveda.
\newblock Approximating {L}iouville measure using local sets of the {G}aussian
  free field.
\newblock {\em arXiv preprint arXiv:1701.05872}, 2017.

\bibitem[Aru15a]{JA}
Juhan Aru.
\newblock {\em The geometry of the Gaussian free field combined with {SLE
  processes and the KPZ} relation}.
\newblock PhD thesis, Ecole normale sup{\'e}rieure de lyon-ENS LYON, 2015.

\bibitem[Aru15b]{JA2}
Juhan Aru.
\newblock {KPZ} relation does not hold for the level lines and {SLE}{$_\kappa
  $} flow lines of the {G}aussian free field.
\newblock {\em Probability Theory and Related Fields}, 163(3-4):465--526, 2015.

\bibitem[AS18]{AS}
Juhan Aru and Avelio Sep{\'u}lveda.
\newblock Survey in preparation.
\newblock 2018.

\bibitem[ASW17]{ASW2}
Juhan Aru, Avelio Sep{\'u}lveda, and Wendelin Werner.
\newblock On bounded-type thin local sets of the two-dimensional {G}aussian
  free field.
\newblock {\em Journal of the Institute of Mathematics of Jussieu}, pages
  1--28, 2017.

\bibitem[BDG01]{Bol-Er-Deu}
Erwin Bolthausen, Jean-Dominique Deuschel, and Giambattista Giacomin.
\newblock Entropic repulsion and the maximum of the two-dimensional harmonic
  crystal.
\newblock {\em Annals of probability}, pages 1670--1692, 2001.

\bibitem[Coo09]{Co}
John~D Cook.
\newblock Upper and lower bounds for the normal distribution function.
\newblock 2009.

\bibitem[DS11]{DS}
Bertrand Duplantier and Scott Sheffield.
\newblock Liouville quantum gravity and {KPZ}.
\newblock {\em Inventiones mathematicae}, 185(2):333--393, 2011.

\bibitem[Dub09]{Dub}
Julien Dub{\'e}dat.
\newblock {S}{L}{E} and the free field: partition functions and couplings.
\newblock {\em Journal of the American Mathematical Society}, 22(4):995--1054,
  2009.

\bibitem[HMP10]{Hu-Mi-Pe}
Xiaoyu Hu, Jason Miller, and Yuval Peres.
\newblock Thick points of the {G}aussian free field.
\newblock {\em The Annals of Probability}, 38(2):896--926, 2010.

\bibitem[LJ11]{LJ}
Yves Le~Jan.
\newblock {\em Markov Paths, Loops and Fields: {\'E}cole d'{\'E}t{\'e} de
  Probabilit{\'e}s de Saint-Flour XXXVIII--2008}, volume 2026.
\newblock Springer, 2011.

\bibitem[Lup16]{Titus}
Titus Lupu.
\newblock From loop clusters and random interlacements to the free field.
\newblock {\em The Annals of Probability}, 44(3):2117--2146, 2016.

\bibitem[MP10]{MoPe}
Peter M{\"o}rters and Yuval Peres.
\newblock {\em Brownian motion}, volume~30.
\newblock Cambridge University Press, 2010.

\bibitem[MS11]{MS}
Jason Miller and Scott Sheffield.
\newblock The {GFF} and {CLE}(4), 2011.
\newblock Slides, talks and private communications.

\bibitem[MS15]{MSQLE2}
Jason Miller and Scott Sheffield.
\newblock Liouville quantum gravity and the {Brownian map I: The QLE }(8/3, 0)
  metric.
\newblock {\em arXiv preprint arXiv:1507.00719}, 2015.

\bibitem[MS16a]{MS1}
Jason Miller and Scott Sheffield.
\newblock Imaginary geometry {I}: interacting {SLE}s.
\newblock {\em Probability Theory and Related Fields}, 164(3-4):553--705, 2016.

\bibitem[MS16b]{MSQLE}
Jason Miller and Scott Sheffield.
\newblock Quantum {L}oewner evolution.
\newblock {\em Duke Mathematical Journal}, 165(17):3241--3378, 2016.

\bibitem[NW11]{NW}
{\c{S}}erban Nacu and Wendelin Werner.
\newblock Random soups, carpets and fractal dimensions.
\newblock {\em Journal of the London Mathematical Society}, 83(3):789--809,
  2011.

\bibitem[QW15]{QW}
Wei Qian and Wendelin Werner.
\newblock Decomposition of {B}rownian loop-soup clusters.
\newblock {\em arXiv preprint arXiv:1509.01180}, 2015.

\bibitem[Roz82]{Roz}
Yu~A Rozanov.
\newblock {\em Markov random fields}.
\newblock Springer, 1982.

\bibitem[She07]{She}
Scott Sheffield.
\newblock Gaussian free fields for mathematicians.
\newblock {\em Probability theory and related fields}, 139(3):521--541, 2007.

\bibitem[SS13]{SchSh2}
Oded Schramm and Scott Sheffield.
\newblock A contour line of the continuum {G}aussian free field.
\newblock {\em Probability Theory and Related Fields}, 157(1-2):47--80, 2013.

\bibitem[SSW09]{SchShW}
Oded Schramm, Scott Sheffield, and David~B Wilson.
\newblock Conformal radii for conformal loop ensembles.
\newblock {\em Communications in Mathematical Physics}, 288(1):43--53, 2009.

\bibitem[Wer16]{WWln2}
Wendelin Werner.
\newblock Topics on the {GFF and CLE(4)}, 2016.

\end{thebibliography}
\bibliographystyle{alpha}

\end{document}